\documentclass[reqno,centertags]{amsart}
\usepackage{amsmath}
\usepackage{amssymb}
\usepackage{amsfonts}

\newtheorem{theorem}{Theorem}
\theoremstyle{plain}

\newtheorem{corollary}{Corollary}

\newtheorem{problem}{Problem}

\numberwithin{equation}{section}

\input{tcilatex}

\begin{document}
\title[]{Surfaces of revolution of finite $III$-type }
\author{Hassan Al-Zoubi*}
\address{Department of Mathematics, Al-Zaytoonah University of Jordan, P.O. Box 130, Amman, Jordan 11733}
\email{dr.hassanz@zuj.edu.jo*}
\author{Tareq Hamadneh}
\address{Department of Mathematics, Al-Zaytoonah University of Jordan, P.O.
Box 130, Amman, Jordan 11733}
\email{t.hamadneh@zuj.edu.jo}
\date{}

\begin{abstract}
In this paper, we consider surfaces of revolution in the 3-dimensional Euclidean space $\mathbb{E}^{3}$ with nonvanishing Gauss curvature. We introduce the finite Chen type surfaces with respect to the third fundamental form of the surface. We present a special case of this class of surfaces of revolution in $\mathbb{E}^{3}$, namely, surfaces of revolution where the sum of the radii of the principal curvature $R$ is constant.
\end{abstract}

\maketitle

\section{Introduction}
In 1983 B. Y. Chen introduces the notion of Euclidean immersions of finite type, \cite{C3}, \cite{C14} and from that time on the research into surfaces of finite type has grown up as one can see in the literature in this area. Many results in this field have been collected in \cite{C7}. In this respect, S. Stamatakis and H. Al-Zoubi restored attention to this theme by introducing the notion of surfaces of finite type corresponding to the second or third fundamental forms in the following way \cite{S1}:

A surface $S$ is said to be of finite type corresponding to the fundamental form $J$, or briefly of finite $J$-type, where $J = II, III$, if the position vector $\boldsymbol{x}$ of $S$ can be written as a finite sum of nonconstant eigenvectors of the Laplacian $\Delta ^{J}$, that is,
\begin{equation}  \label{1}
\boldsymbol{x}=\boldsymbol{c}+\sum_{i=1}^{k}\boldsymbol{x}_{i},\quad
\Delta ^{J}\boldsymbol{x}_{i}=\lambda _{i}\,\boldsymbol{x}_{i},\quad
i=1,\dotsc ,k,
\end{equation}%
where $\boldsymbol{c}$ is a fixed vector and $\lambda _{1},\lambda_{2},\dotsc ,\lambda _{k}$ are eigenvalues of the operator $\Delta^{J}$. In particular, if all eigenvalues $\lambda _{1},\lambda _{2},\dotsc ,\lambda_{k}$ are mutually distinct, then $S$ is said to be of finite $J$-type $k$. When $\lambda _{i}=0$ for some $i=1,\dotsc ,k$, then $S$ is said to be of finite null $J$-type $k$. Otherwise, $S$ is said to be of infinite type.

In general when $S$ is of finite type $k$, it follows from (\ref{1}) that there exists a monic polynomial, say $F(x)\neq 0,$ such that $F(\Delta ^{J})(\boldsymbol{x}-\boldsymbol{c})=\mathbf{0}.$ Suppose that $F(x)=x^{k}+\sigma_{1}x^{k-1}+...+\sigma _{k-1}x+\sigma _{k},$ then coefficients $\sigma _{i}$ are given by

\begin{eqnarray}
\sigma _{1} &=&-(\lambda _{1}+\lambda _{2}+...+\lambda _{k}),  \notag \\
\sigma _{2} &=&(\lambda _{1}\lambda _{2}+\lambda _{1}\lambda_{3}+...+\lambda_{1}\lambda _{k}+\lambda _{2}\lambda _{3}+...+\lambda _{2}\lambda
_{k}+...+\lambda _{k-1}\lambda _{k}),  \notag \\
\sigma _{3} &=&-(\lambda _{1}\lambda _{2}\lambda _{3}+...+\lambda_{k-2}\lambda _{k-1}\lambda _{k}),  \notag \\
&&.............................................  \notag \\
\sigma _{k} &=&(-1)^{k}\lambda _{1}\lambda _{2}...\lambda _{k}.  \notag
\end{eqnarray}
Therefore, the position vector $\boldsymbol{x}$ satisfies the following equation (see \cite{A3})

\begin{equation*}  \label{2}
(\Delta^{J})^{k}\boldsymbol{x}+\sigma_{1}(\Delta^{I})^{k-1}\boldsymbol{x}+...+\sigma_{k}( \boldsymbol{x}-\boldsymbol{x}_{0})=\boldsymbol{0}.
\end{equation*}

Surfaces of finite type with respect to the second or third fundamental forms became a topic of active research in the last years. In this paper we will focus on surfaces which are of finite type corresponding to the third fundamental form only.

Up to now the only known surfaces of finite $III$-type are \cite{S1}

\begin{itemize}
\item the spheres, which actually are of finite $III$-type 1,
\item the minimal surfaces, which are of finite null $III$-type 1, and
\item the parallel surfaces to the minimal surfaces, which in fact are of finite null $III$-type 2.
\end{itemize}

So, we raise the following question which seems to be interesting:
\begin{problem}  \label{(p1)}
Determine all surfaces of finite $III$-type in the Euclidean 3-space.
\end{problem}

With the aim of getting an answer to the above problem, it is worthwhile investigating the classification of surfaces in the Euclidean space $\mathbb{E}^{3}$ in terms of finite $III$-type, by studying important families of surfaces. More precisely, ruled surfaces \cite{A4}, tubes \cite{A5} and quadrics \cite{A6} are surfaces of the only known examples in $\mathbb{E}^{3}$. However, for other classical families of surfaces, such as surfaces of revolution, translation surfaces, quadric surfaces, cyclides of Dupin, spiral surfaces and helicoidal surfaces, the classification of its finite $III$-type surfaces is not known yet.

On the other hand, one can generalize the problem mentioned above by studying surfaces in $\mathbb{E}^{3}$ whose position vector $\boldsymbol{x}$ satisfies the following condition
\begin{equation}
\Delta ^{III}\boldsymbol{x} =A\boldsymbol{x},  \label{3}
\end{equation}%
where $A\in \mathbb{Re}^{3\times 3}$.

In view of this, surfaces satisfying condition (\ref{3}) are said to be of coordinate finite $III$-type \cite{D3}. Here again, we also
pose the following problem
\begin{problem}  \label{(p2)}
Determine all surfaces in the Euclidean 3-space of coordinate finite $III$-type.
\end{problem}

In connection with this problem, it was shown that the only surfaces of revolution satisfying condition (\ref{3}) are the spheres and the catenoids \cite{S2}. In \cite{A1} this problem was studied for two classes of surfaces, namely, ruled surfaces and quadrics. In fact authored ascertained that helicoids are the only ruled surfaces in the 3-dimensional Euclidean space that satisfy (\ref{3}), meanwhile, spheres are the only quadric surfaces that satisfy (\ref{3}). Next, in \cite{A2} it is proved that Scherk's surface is the only translation surface in the 3-dimensional Euclidean space that satisfies (\ref{3}).

In order to achieve our goal, we introduce a formula for $\Delta ^{III}\boldsymbol{x}$ and $\Delta ^{III}\boldsymbol{n}$ by using tensors calculations in the second section, where $\boldsymbol{n}$ denotes the Gauss map of $S$. Further, in the last section, we continue our study by proving finite type surfaces for an important class of surfaces, namely, surfaces of revolution in the Euclidean 3-space of which the sum of the radii of the principal curvature is constant.

\section{Basic concepts}
Let $S$ be a smooth surface in $\mathbb{E}^{3}$ given by a patch $\boldsymbol{x} = \boldsymbol{x}(u^{1}, u^{2})$ on a region $U: = I \times \mathbb{R} \,\,\,(I\subset \mathbb{R}$ open interval) of $\mathbb{R}^{2}$ whose Gaussian curvature never vanishes. 
We denote by
\begin{equation}  \label{20}
I = g_{ij}du^{i}du^{j},\ \ \    II = b_{ij}du^{i}du^{j}, \ \ \   III = e_{ij}du^{i}du^{j}
\end{equation}
the first, second and third fundamental forms of $S$ respectively.
For two sufficiently differentiable functions $f(u^{1}, u^{2})$ and $g(u^{1}, u^{2})$ on $S$ the first differential parameter of Beltrami with respect to the fundamental form $J$, where $J = I, II, III$ is defined by \cite{H1}
\begin{equation}  \label{21}
\nabla^{J}(f,g)=a^{ij}f_{/i}g_{/j},
\end{equation}
where $f_{/i}:=\frac{\partial f}{\partial u^{i}}$ and $a^{ij}$ denote the components of the inverse tensor of $g_{ij}, b_{ij}, e_{ij}$ for $J = I, II$ and $III$ respectively. The second Beltrami-Laplace operator with respect to the fundamental form $J$ of $S$ is defined by

\begin{equation}  \label{22}
\Delta ^{J}f =-a^{ij}\nabla^{J}_{i} f_{j},
\end{equation}
where $f$ is sufficiently differentiable function, $\nabla^{J}$ is the covariant derivative in the $u^{i}$ direction with respect to the fundamental form $J$ and $(a^{ij})$ stands as in definition (\ref{21}) for the inverse tensor of $(g_{ij}), (b_{ij})$ and $(e_{ij})$ for $J = I, II$ and $III$ respectively \cite{A7}.

Applying (\ref{22}) for the position vector $\boldsymbol{x}$ corresponding to the third fundamental form we find \cite{S1}
\begin{equation}  \label{23}
\triangle^{III}\boldsymbol{x}=grad^{III}R- R\boldsymbol{n}.
\end{equation}




By applying (\ref{22}) for the normal vector $\boldsymbol{n}$ we get \cite{S1}
\begin{equation}  \label{24}
\Delta ^{III}\boldsymbol{n}=2\boldsymbol{n},
\end{equation}

From (\ref{24}) we have \cite{S1}
\begin{corollary}
The Gauss map of every surface $S$ in $\mathbb{E}^{3}$ is of finite $III$-type 1. The corresponding eigenvalue is $\lambda = 2$.
\end{corollary}

We consider now the surface $S$ of finite $III$-type 1. Then we have $\triangle^{III}\boldsymbol{x}=\lambda\boldsymbol{x}$. From (\ref{23}) we get $grad^{III}R- R\boldsymbol{n}= \lambda\boldsymbol{x}$. Taking the inner product of both sides of this equation with $\boldsymbol{n}$ we find $R= \lambda w$, where $w:=-<\boldsymbol{n},\boldsymbol{x}>$ is the support function. From the well known formula
\begin{equation*}
-\triangle^{III}w +2w=R,
\end{equation*}
we find
\begin{equation*}
\triangle^{III}w=(2-\lambda)w,\ \ \ \ \triangle^{III}R=(2-\lambda)R.
\end{equation*}

Thus, we also have the following \cite{A6}:
\begin{theorem}
\label{T7} Let $S$ be a surface in $\mathbb{E}^{3}$ of finite $III$-type 1 with corresponding eigenvalue $\lambda$. Then the support function $w$ and the sum of the principal radii of curvature $R$ are of eigenfunctions of the Laplacian $\triangle^{III}$ with corresponding eigenvalue $2-\lambda$.
\end{theorem}

Let $S^{*}$ be a parallel surface of $S$ in (directed) distance $\mu$ = const. $\neq 0$, so that $1-2\mu H+\mu^{2}K \neq 0$. Then $S^{*}$ possesses the position vector
\begin{equation*}  \label{24.1}
\boldsymbol{x}^{*} = \boldsymbol{x}+\mu \boldsymbol{n}.
\end{equation*}

Denoting by $H^{*}$ the mean and by $K^{*}$ the Gauss curvature of $S^{*}$ we mention the following relations \cite{H1}
\begin{equation*}  \label{24.2}
K^{*}=\frac{K}{1-2\mu H+\mu^{2}K},
\end{equation*}
\begin{equation*}  \label{24.3}
H^{*}=\frac{H-\mu K}{1-2\mu H+\mu^{2}K},
\end{equation*}

Hence we get
\begin{equation}  \label{24.4}
R^{*}=\frac{2H^{*}}{K^{*}} = R-\mu .
\end{equation}

Furthermore the surfaces $S$, $S^{*}$ have common unit normal vector and spherical image. Thus $III = III^{*}$ and
\begin{equation*}  \label{24.5}
\Delta ^{III^{*}} = \Delta ^{III}.
\end{equation*}

We prove now the following theorem for later use.
\begin{theorem}
\label{T8} Let $S$ be a minimal surface in $\mathbb{E}^{3}$. Then $S^{*}$ is parallel surface of $S$ if and only if the principal radii of curvature $R^{*}$ of $S^{*}$ is constant.
\end{theorem}
\begin{proof}
Suppose that $S$ is minimal surface in $\mathbb{E}^{3}$, which is defined on a simply connected domain $D$ in the $(u^{1}, u^{2})$-plane. Let
\begin{equation*}
S^{*}:\boldsymbol{x}^{*} = \boldsymbol{x}+\mu \boldsymbol{x},\ \ \ \mu \in \mathbb{R},  \ \ \ \mu \neq 0
\end{equation*}
is parallel surface of $S$.

From (\ref{24.4}) and taking into account $H = 0$, we find $R^{*}= -2\mu$ = const. Hence the first part of the theorem is proved.

Conversely, let $R^{*}$ = const. $\neq 0$. Then from Theorem (4.4) (see \cite{S1}), $S^{*}$ is of null $III$- type 2. Therefore from (\ref{1}) there exist nonconstant eigenvectors $\boldsymbol{x_{1}}(u^{1}, u^{2})$ and $\boldsymbol{x_{2}}(u^{1}, u^{2})$ defined on the same domain $D$ such that
\begin{equation}  \label{24.51}
\boldsymbol{x}^{*}= \boldsymbol{x_{1}} + \boldsymbol{x_{2}},
\end{equation}
where
\begin{equation*}
\Delta ^{III}\boldsymbol{x_{1}} = \lambda_{1}\boldsymbol{x_{1}},\ \ \ \Delta ^{III}\boldsymbol{x_{2}} = \lambda_{2}\boldsymbol{x_{2}},
\end{equation*}
and here we have $\lambda_{1}$ = 0 since $S^{*}$ is of null $III$- type 2.

Once we have $\Delta ^{III}\boldsymbol{x}^{*}= \Delta ^{III}\boldsymbol{x_{1}} + \Delta ^{III}\boldsymbol{x_{2}}$, it then follows that
\begin{equation}  \label{24.6}
\Delta ^{III}\boldsymbol{x}^{*}= \lambda_{2}\boldsymbol{x_{2}}.
\end{equation}

On the other hand, since $R^{*}$ = const. $\neq 0$, we get
\begin{equation}  \label{24.7}
\Delta ^{III}\boldsymbol{x}^{*}= -R^{*}\boldsymbol{n}.
\end{equation}

Thus from (\ref{24.6}) and (\ref{24.7}), one finds
\begin{equation*}  \label{24.8}
\lambda_{2}\boldsymbol{x_{2}}= -R^{*}\boldsymbol{n},
\end{equation*}
or
\begin{equation*}  \label{24.8}
\boldsymbol{x_{2}}=c\boldsymbol{n},
\end{equation*}
where $c =\frac{-R^{*}}{\lambda_{2}}$, and then (\ref{24.51}) becomes
\begin{equation}  \label{24.9}
\boldsymbol{x}^{*}= \boldsymbol{x_{1}} + c\boldsymbol{n}.
\end{equation}

The differential of the above equation is
\begin{equation}  \label{24.10}
d\boldsymbol{x}^{*}= d\boldsymbol{x_{1}} + cd\boldsymbol{n}.
\end{equation}

Taking the inner product of both sides of (\ref{24.10}) with $\boldsymbol{n}$ yields
\begin{equation}  \label{24.11}
<d\boldsymbol{x_{1}} ,\boldsymbol{n}> = 0.
\end{equation}

Now we want to show that $\boldsymbol{x_{1}}(u^{1}, u^{2})$ is a regular parametric representation of a surface in $\mathbb{E}^{3}$. It is enough to prove that
\begin{equation}  \label{24.12}
\frac{\partial\boldsymbol{x_{1}}}{\partial u^{1}} \times \frac{\partial\boldsymbol{x_{1}}}{\partial u^{2}} \neq \boldsymbol{0}
\end{equation}
for each $(u^{1}, u^{2})\in D$, where $\times$ is the Euclidean cross product. We have
\begin{equation}  \label{24.13}
\boldsymbol{x_{1}} = \boldsymbol{x^{*}} - \mu\boldsymbol{n}.
\end{equation}
Using the Wingarten equations \cite{H1}
\begin{equation*}  \label{24.131}
\frac{\partial \boldsymbol{n}}{\partial u^{i}} = -b_{ij}g^{jr}\frac{\partial \boldsymbol{x^{*}}}{\partial u^{r}}
\end{equation*}
and the equation (\ref{24.13}), it follows from (\ref{24.12}) that
\begin{equation}  \label{24.14}
\big[\frac{\partial \boldsymbol{x^{*}}}{\partial u^{1}}-\mu \frac{\partial \boldsymbol{n}}{\partial u^{1}}\big] \times \big[\frac{\partial \boldsymbol{x^{*}}}{\partial u^{2}}-\mu \frac{\partial \boldsymbol{n}}{\partial u^{2}}\big] = (1-2\mu H+\mu^{2}K)(\frac{\partial\boldsymbol{x^{*}}}{\partial u^{1}} \times \frac{\partial\boldsymbol{x^{*}}}{\partial u^{2}}) \neq \boldsymbol{0}.
\end{equation}

Hence, on account of (\ref{24.14}) and (\ref{24.11}), we conclude that $\boldsymbol{x_{1}}(u^{1}, u^{2})$ is a regular parametric representation of a surface in $\mathbb{E}^{3}$ with $\boldsymbol{n}$ it's Gauss map.

Since $\Delta ^{III}\boldsymbol{x_{1}}=\boldsymbol{0}$. Consequently, from Theorem (3.1) (see \cite{S1}), $\boldsymbol{x_{1}}(u^{1}, u^{2})$ is a minimal surface. Thus from (\ref{24.9}) we obtain that $S^{*}$ is parallel surface of a minimal.
\end{proof}
Now we prove our main result.
\section{Proof of the main theorem}

A surface of revolution is formed by revolving a plane curve about a line in $\mathbb{E}^{3}$. Let E be a plane in $\mathbb{E}^{3}$, let $l$ and $C$ be a line and a point set of a plane curve which does not intersect $l$ in $E$, respectively. When $C$ is revolved in $\mathbb{E}^{3}$ about $l$, the resulting point set $S$ is called the surface of revolution generated by $C$. In this case, $C$ is called the profile curve of $S$ and the line $l$ is called the axis of revolution of $S$. For convenience we choose $E$ to be the $xz$-plane and $l$ to be the $z$-axis. We shall assume that the point set $C$ has a parametrization
\begin{equation*}
\varrho : J = (a, b) \longrightarrow C
\end{equation*}
defined by $u \longrightarrow (f(u), 0, g(u))$, which is differentiable \cite{K8}. Without loss of generality, we can assume that $f(u)$ is a positive function, and $g$ is a function on $J$. On the other hand, a subgroup of the rotation group which fixes the vector $(0, 0, 1)$ is generated by
\begin{equation*}
\left(
\begin{array}{ccc}
\cos v & 0-\sin v & 0 \\
\sin v & \cos v & 0 \\
0 & 0& 1
\end{array}%
\right),
\end{equation*}
for any $v \in [0,2\pi) $ . Hence the surface $S$ of revolution can be written as
\begin{equation*}
\boldsymbol{x}(u,v)=\left(
\begin{array}{ccc}
\cos v & -\sin v & 0 \\
\sin v & \cos v & 0 \\
0 & 0& 1
\end{array}%
\right) \left(
\begin{array}{ccc}
f(u)  \\
0  \\
g(u)
\end{array}%
\right),
\end{equation*}
or
\begin{equation}  \label{31}
\boldsymbol{x}(u,v)= \big(f(u)\cos v, f(u)\sin v, g(u) \big), \ \ \ u\in J, \ \ \ v \in [0,2\pi)
\end{equation}
(For the parametric representation of surfaces of revolution, we also refer the reader to \cite{D4}).

Without loss of generality, we may assume that $\varrho$ has the arc-length parametrization, i.e., it satisfies

\begin{equation}  \label{32}
(f\prime)^{2}+ (g\prime)^{2}=1
\end{equation}
where $\prime := \frac{d}{du}$.
Furthermore it is $f\prime g\prime \neq 0$, because otherwise $f$ = const. or $g$ = const. and $S$ would be a circular cylinder or part of a plane, respectively. Hence $S$ would consist only of parabolic points, which has been excluded.

Using the natural frame $\{{\boldsymbol{x}_{u}, \boldsymbol{x}_{v}}\}$ of $S$ defined by
\begin{equation*}
\boldsymbol{x_{u}}=\left( f\prime(u)\cos v, f\prime(u)\sin v, g\prime(u)\right),
\end{equation*}
and
\begin{equation*} \newline
\boldsymbol{x_{v}}=\left( -f(u)\sin v, f(u)\cos v, 0\right),
\end{equation*}

\noindent the components $g_{ij}$ of the first fundamental form in (local) coordinates are the following
\begin{equation*}
g_{11}= <\boldsymbol{x_{u}}, \boldsymbol{x_{u}}> = 1,\ \ \ g_{12}= <\boldsymbol{x_{u}}, \boldsymbol{x_{v}}> = 0,\ \ \ g_{22}= <\boldsymbol{x_{v}}. \boldsymbol{x_{v}}>= f^{2},
\end{equation*}

Denoting by $R_{1}, R_{_{2}}$ the principal radii of curvature of $S$ and $\kappa$ the curvature of the curve $C$, we have
\begin{equation*}  \label{33}
R_{1} = \kappa, \ \ \ \ R_{2}= \frac{g\prime}{f}.
\end{equation*}

The mean curvature and the Gauss curvature of $S$ are respectively
\begin{align*}  \label{34}
2H = R_{1}+R_{2} = \kappa+ \frac{g\prime}{f},\ \ \ \ K = R_{1}R_{2} =\frac{\kappa g\prime}{f}= -\frac{f\prime\prime}{f}.
\end{align*}

The Gauss map $\boldsymbol{n}$ of $S$ is given by
\begin{equation}  \label{35}
\boldsymbol{n}(u,v)= \frac{\boldsymbol{x_{u}}\times\boldsymbol{x_{v}}}{\sqrt{\mathfrak{g}}} = -\big(g\prime\cos v, g\prime\sin v, f\prime \big),
\end{equation}
where $\mathfrak{g}: = \det (g_{ij})$.
Now, by using the natural frame $\{{\boldsymbol{n}_{u}, \boldsymbol{n}_{v}}\}$ of $S$ defined by
\begin{equation*}
\boldsymbol{n_{u}}=-\big(g\prime\prime\cos v, g\prime\prime\sin v, f\prime\prime \big),
\end{equation*}
and
\begin{equation*} \newline
\boldsymbol{n_{v}}=-\big(-g\prime\sin v, g\prime\cos v, 0 \big),
\end{equation*}

\noindent the components $e_{ij}$ of the third fundamental form in (local) coordinates are the following
\begin{equation*}
e_{11}=<\boldsymbol{n_{u}}, \boldsymbol{n_{u}}> = (g\prime\prime)^{2}+(f\prime\prime)^{2},\ \ \ e_{12}=<\boldsymbol{n_{u}}, \boldsymbol{n_{v}}> = 0,\ \ \ e_{22}=<\boldsymbol{n_{v}}, \boldsymbol{n_{v}}> = (g\prime)^{2}.
\end{equation*}

The Beltrami operator $\Delta ^{III}$ of the third fundamental form in terms of local coordinates $(u, v)$ of $S$ can be expressed as follows

\begin{eqnarray}  \label{36}
\Delta ^{III} =-\frac{1}{\kappa^{2}}\frac{\partial^{2}}{\partial u^{2}}-\frac{1}{(g\prime)^{2}} \frac{\partial^{2}}{\partial v^{2}}
+\frac{g\prime\kappa\prime-\kappa g\prime\prime}{\kappa^{3}g\prime}\frac{\partial}{\partial u}.
\end{eqnarray}

In view of (\ref{32}) we can put
\begin{equation}  \label{37}
f\prime= \cos\varphi, \ \ \ g\prime = \sin\varphi,
\end{equation}
where $\varphi$ is a function of $u$. Then $\kappa = \varphi\prime$ and the unit normal vector $\boldsymbol{n}$ of $S$ is given by

\begin{equation}   \label{38}
\boldsymbol{n}(u,v)= -\big(\sin\varphi\cos v, \sin\varphi\sin v, \cos\varphi \big).
\end{equation}

Relation (\ref{36}) becomes
\begin{eqnarray}  \label{39}
\Delta ^{III} =-\frac{1}{(\varphi\prime^){2}}\frac{\partial^{2}}{\partial u^{2}}-\frac{1}{\sin^{2}\varphi} \frac{\partial^{2}}{\partial v^{2}}
+\frac{\varphi\prime\prime}{(\varphi\prime)^{3}}-\frac{\cos\varphi}{\varphi\prime\sin\varphi}\frac{\partial}{\partial u}.
\end{eqnarray}

For the sum of the principal radii of curvature $R = \frac{1}{R_{1}}+\frac{1}{R_{2}} =\frac{2H}{K}$, one finds
\begin{equation}   \label{40}
R= \frac{1}{\varphi\prime}+\frac{f}{\sin\varphi}.
\end{equation}

On differentiating relation (\ref{40}) we obtain
\begin{equation}   \label{41}
R\prime= -\frac{\varphi\prime\prime}{(\varphi\prime)^{2}}-\frac{f\varphi\prime\cos\varphi}{\sin^{2}\varphi}+\frac{\cos\varphi}{\sin\varphi}.
\end{equation}
Let $(x_{1}, x_{2}, x_{3})$ be the coordinate functions of the parametric representation (\ref{31}). By virtue of (\ref{39}), one can find

\begin{equation}  \label{42}
\Delta^{III}x_{1} =\Delta^{III}(f\cos v) = \bigg(\frac{\varphi\prime\prime\cos\varphi}{(\varphi\prime)^{3}}+\frac{2\sin\varphi}{\varphi\prime} -\frac{1}{\varphi\prime\sin\varphi}+\frac{f}{\sin^{2}\varphi}\Bigg)\cos v,
\end{equation}
\begin{equation}  \label{43}
\Delta^{III}x_{2} =\Delta^{III}(f\sin v) = \bigg(\frac{\varphi\prime\prime\cos\varphi}{(\varphi\prime)^{3}}+\frac{2\sin\varphi}{\varphi\prime} -\frac{1}{\varphi\prime\sin\varphi}+\frac{f}{\sin^{2}\varphi}\Bigg)\sin v,
\end{equation}
\begin{equation}  \label{44}
\Delta^{III}x_{3} =\Delta^{III}(g) = \frac{\varphi\prime\prime\sin\varphi}{(\varphi\prime)^{3}}-\frac{2\cos\varphi}{\varphi\prime},
\end{equation}

From (\ref{40}) and (\ref{41}), equations (\ref{42}), (\ref{43}) and (\ref{44}) become respectively
\begin{equation}  \label{45}
\Delta^{III}x_{1} = \bigg(R\sin\varphi-\frac{R\prime\cos\varphi}{\varphi\prime}\Bigg)\cos v,
\end{equation}
\begin{equation}  \label{46}
\Delta^{III}x_{2} = \bigg(R\sin\varphi-\frac{R\prime\cos\varphi}{\varphi\prime}\Bigg)\sin v,
\end{equation}
\begin{equation}  \label{47}
\Delta^{III}x_{3} = -R\cos\varphi-\frac{R\prime\sin\varphi}{\varphi\prime}.
\end{equation}

We distinguish two cases:

\textit{Case I.}  $R \equiv 0$.

Then $H \equiv 0$. Consequently $S$, being a minimal surface of revolution, is a catenoid \cite{S2}.

\textit{Case II.}  $R = const. \neq 0$.

From (\ref{45}), (\ref{46}) and (\ref{47}) we obtain
\begin{equation}  \label{49}
\Delta^{III}x_{1} = R\sin\varphi\cos v,
\end{equation}
\begin{equation}  \label{50}
\Delta^{III}x_{2} = R\sin\varphi\sin v,
\end{equation}
\begin{equation}  \label{51}
\Delta^{III}x_{3} = -R\cos\varphi.
\end{equation}

Let $(n_{1}, n_{2}, n_{3})$ be the coordinate functions of the parametric representation (\ref{38}). From (\ref{38}), relations (\ref{49}), (\ref{50}) and (\ref{51}) can be written
\begin{equation}  \label{52}
\Delta^{III}x_{1} = -Rn_{1},\ \ \ \Delta^{III}x_{2} = -Rn_{2},\ \ \ \Delta^{III}x_{3} = -Rn_{3},\ \ \
\end{equation}
and hence
\begin{equation}  \label{53}
\Delta^{III}\boldsymbol{x}= -R\boldsymbol{n}.
\end{equation}

In view of (\ref{24}) and (\ref{53}) we have
\begin{equation}
(\Delta^{III})^{r}\boldsymbol{x}= -(2^{r-1})R\boldsymbol{n}.
\end{equation}  \label{54}

Now, if $S$ is of finite type, then there exist real numbers, $c_{1}, c_{2}, ... , c_{m}$ such that

\begin{equation}  \label{55}
(\Delta ^{III})^{m}\boldsymbol{x}+c_{1}(\Delta ^{III})^{m-1}\boldsymbol{x}+...+c_{m-1}\Delta ^{III}\boldsymbol{x}+c_{m}\boldsymbol{x}=\mathbf{0}.
\end{equation}

From (\ref{53}) and (3.21), equation (\ref{55}) becomes
\begin{equation}  \label{56}
-2^{m-1}R\boldsymbol{n}-2^{m-2}c_{1}R\boldsymbol{n}-...-c_{m-1}R\boldsymbol{n}+c_{m}\boldsymbol{x}= \boldsymbol{0}
\end{equation}
or
\begin{eqnarray}  \label{56}
c\boldsymbol{n} + c_{m}\boldsymbol{x}= \boldsymbol{0},
\end{eqnarray}
where $c = -(2^{m-1}+2^{m-2}c_{1}+...+c_{m-1})R = constant$.

Now, if $c_{m}\neq 0$, then (\ref{56}) yields $\boldsymbol{x} = -\frac{c}{c_{m}}\boldsymbol{n}$, from which we have $|\boldsymbol{x}| =|\frac{c}{c_{m}}|$ and hence $S$ is a sphere. Thus according to Theorem (3.3) (see \cite{S1}), $S$ is of finite $III$-type 1.
If $c_{m} = 0$, then $S$ is of 0- type $m$. Since $R = const.$, therefore from Theorem (4.4) (see \cite{S1}) and Theorem (\ref{T8}), $S$ is of null $III$-type
2 which is parallel surface of a minimal.






\end{document}